\documentclass{amsart}

\usepackage[active]{srcltx}  
\usepackage{amsmath}
\usepackage{amsfonts,amssymb,amsthm,amscd,latexsym,euscript}

\usepackage[all]{xy}

\newcommand{\Sp}{\ensuremath{\textup{Sp}}}

\newcommand{\Grp}{\ensuremath{\textup{Grp}}}
\newcommand{\Set}{\ensuremath{\textup{Set}}}
\newcommand{\op}{{\ensuremath{\textup{op}}}}

\newcommand {\cofib} {\ensuremath{\hookrightarrow}}

\newcommand {\trivcofib} {\ensuremath{\tilde\hookrightarrow}}

\newcommand {\trivfibr} {\ensuremath{\tilde\twoheadrightarrow}}

\newcommand{\cal}[1]{\ensuremath{\mathcal #1}}

\newtheorem {theorem1}{Theorem}[section]
\newtheorem {theorem}[theorem1]{Theorem}

\newtheorem {proposition}[theorem1]{Proposition}
\newtheorem {lemma}[theorem1]{Lemma}
\newtheorem {condition}[theorem1]{Condition}
\theoremstyle{definition}
\newtheorem {definition}[theorem1]{Definition}

\theoremstyle{remark}
\newtheorem {remark}[theorem1]{Remark}

\newcommand{\calT}{\ensuremath{\mathcal{T}} }

\newcommand{\cat}[1]{\ensuremath{\EuScript #1}}

\newcommand{\sS}{\ensuremath{\mathcal{S}} }

\newcommand{\Top}{\ensuremath{\calT\textup{op}}}

\newcommand{\colim}{\ensuremath{\mathop{\textup{colim}}}}
\newcommand{\hocolim}{\ensuremath{\mathop{\textup{hocolim}}}}

\DeclareMathOperator{\Lan}{\ensuremath{\textup{Lan}}}
\DeclareMathOperator{\Ran}{\ensuremath{\textup{Ran}}}

\newcommand{\Id}{\ensuremath{\textup{Id}}}

\renewcommand{\hom}{\ensuremath{{\rm hom}}}

\newcounter{zahl}%
    {\end{list}}%

\begin{document}

\SelectTips{cm}{10}

\title{Representability theorems, up to homotopy}
\author{David Blanc}\thanks{The first author acknowledges the support of ISF 770/16 grant.}
\author{Boris Chorny}\thanks{The second author acknowledges the support of ISF 1138/16 grant.}

\address{Department of Mathematics, University of Haifa, Haifa, Israel}
\email{blanc@math.haifa.ac.il}

\address{Department of Mathematics, University of Haifa at Oranim, Tivon, Israel}
\email{chorny@math.haifa.ac.il}



\date{\today}
\dedicatory{}
\commby{}

\begin{abstract}
  We prove two representability theorems, up to homotopy, for presheaves taking
  values in a closed symmetric combinatorial model category \cat V.
  The first theorem resembles the Freyd representability theorem,
  the second theorem is closer to the Brown representability theorem.
  As an application we discuss a recognition principle for mapping spaces.
\end{abstract}

\maketitle

\section{Introduction}
It is a classical question in homotopy theory whether for a given space $X\in \Top_*$
there exists a space $Y\in \Top_*$ such that $X\simeq \hom(S^1, Y)$.
Several solutions to this question have emerged, beginning with Sugawara's work
in \cite{Sugawara}. The approach proposed by Stasheff in \cite{Stasheff},
Boardman and Vogt in \cite{Boardman-Vogt}, and May in \cite{May} are known nowadays as
operadic, while Segal's loop space machine (see \cite{Segal}), is closer to Lawvere's
notion of an algebraic theory. Later on, the canonical delooping machine by Badzioch, Chung,
and Voronov (see \cite{canonical}) provided a simplicial algebraic theory
$\cat T_{n}, \,  n\geq 0$ allowing for the recognition of an $n$-fold loop space
as a homotopy algebra over $\cat T_{n}$.

It is natural to ask whether there exist similar recognition principles for mapping
spaces of the form $X = \hom(A,Y)$ for $A\in \Top_*$ other than $S^n$? When $A = S^2\vee S^3$,
for example, there is no simplicial algebraic theory \cat T such that all spaces of the
form $X=\hom (S^2\vee S^3, Y)= \Omega^2 Y\times \Omega^3 Y$ are homotopy algebras
over \cat T (see \cite[p.~2]{canonical}).  More generally, if a space $A$ has rational
homology in more than one positive dimension, there is no simplicial algebraic
theory \cat T such that all the spaces of the form $X=\hom(A,Y)$ have
the structure of homotopy algebras over \cat T, \cite{Badz-Dobr-note}.

In this paper we suggest using a larger category than an algebraic theory
for the recognition of arbitrary mapping spaces, up to homotopy. This category
will be closed under arbitrary homotopy colimits,  unlike an algebraic theory
which is closed only under finite (co)products. The minimal subcategory of spaces
containing $A$ and closed under the homotopy colimits is denoted $C(A)$.
This approach was first introduced by Badzioch, Dorabia{\l}a, and the first author
in \cite{Bad-Blanc-Dor}, where an attempt to limit the homotopy colimits involved was made.
We take a different approach here, and consider the functors defined on the
large subcategory of spaces $C(A)$. Given a space $X\in \Top_*$, suppose
there exists a functor $F\colon C(A)\to \Top_*$ taking homotopy colimits to
homotopy limits and satisfying $F(A)\simeq X$. The question of whether there
exists a space $Y$ such that $X\simeq \hom(A, Y)$ for some $Y$ is equivalent
to the question of representability of the functor $F$, up to homotopy.

Theorems about representability of functors are naturally divided
into two main types: Freyd and Brown representability theorems. In both cases
some exactness condition for the functor under consideration is necessary.
Theorems of Freyd type use a set-theoretical assumption about the functor:
for example, the solution set condition, or accessibility (see \cite[3, Ex.~G,J]{Freyd-book},
\cite[4.84]{Kelly}, and \cite[1.3]{Neeman-Rosicky}). Theorems of Brown type use
set-theoretical assumptions about the domain category, such as the existence
of sufficiently many compact objects (see \cite{Brown}, \cite{Neeman-book},
\cite{Krause}, and \cite{Franke}).

In this paper we address the question of representability (up to homotopy) of functors
taking values in a closed symmetric combinatorial model category. After
some technical preliminaries in Section~\ref{prelim}, we prove the Freyd version of
representability up to homotopy in Section~\ref{Freyd}.
The solution set condition is replaced by the requirement that the functor be small.
Theorem~\ref{Freyd-rep} generalizes \cite[Theorem 4.84]{Kelly} to functors defined
on a $\cat V$-model category, rather than just a $\cat V$-category. At the same time, it
generalizes a result the first author on the representability of small contravariant
functors from spaces to spaces (see \cite{Chorny-BrownRep}).

The Brown version of representability up to homotopy is proved in Section~\ref{Brown}.
The set-theoretical condition concerning the domain category is local presentability.
In other words, we show that for any \cat V-presheaf  $H$ defined on a combinatorial
\cat V-model category \cat M and taking homotopy colimits to homotopy limits,
there is a fibrant object $Y\in \cat M$ and a natural transformation
$h\colon H(-)\to \hom(-,Y)$, which is a weak equivalence for every cofibrant
$X\in \cat M$. A similar theorem for functors taking values in simplicial sets was
proved by Jardine in \cite{Jardine-representability}. However, the conditions
required for the Brown representability, up to homotopy, to hold are formulated
for the homotopy category of the model category and do not allow for an easy
verification in an arbitrary combinatorial model category.

In Section~\ref{counter-example} we provide an example of a non-small presheaf,
defined on a non-combinatorial model category, which is not representable up to homotopy,
This shows that representability theorems are not tautological.

In Section~ref{recognition} we interpret Brown representability up to homotopy
as a recognition principle for mapping spaces for an arbitrary space $A$
(rather than just $S^n$).

\section{(Model) Categorical preliminaries}\label{prelim}

For every closed symmetric monoidal combinatorial model category
\cat V and $\cat V$-model category \cat M (not necessarily
combinatorial), we can consider the category of small presheaves
$\cat V^{\cat M^{\op}}$. This is a \cat V-category of functors,
which are left Kan extensions from small subcategory of \cat M.
The category $\cat V^{\cat M^{\op}}$ is cocomplete by
\cite[5.34]{Kelly}. Since \cat V is a combinatorial model
category, it is in particular locally presentable. Therefore, the
category of small presheaves  $\cat V^{\cat M^{\op}}$ is also
complete by \cite{Day-Lack}. For a \cat V-category \cat C we denote by $\cat C_0$ the
underlying category of $\cat C$ (enriched only in $\Set$).

\begin{definition}
A natural transformation $f\colon F \to G$ in  $\left(\cat V^{\cat M^{\op}}\right)_0$ is called a
\emph{cofibrant-projective weak equivalence} (respectively, a
\emph{cofibrant-projective fibration}) if for all cofibrant $M\in\cat M$,
the induced map $f_M\colon F(M)\to G(M)$ is a weak equivalence (respectively, a fibration).
The notion of a \emph{projective weak equivalence} (respectively, a
\emph{projective fibration}) in $\left(\cat V^{\cat M^{\op}}\right)_0$ is a particular case of the
cofibrant-projective analog, when all objects of \cat M are cofibrant -- e.g., for
the trivial model structure on \cat M. If (cofibrant-)projective fibrations and
weak equivalences give rise to a model structure on the category of small presheaves,
this model structure is called \emph{(cofibrant-)projective}.
\end{definition}

First of all, we would like to establish the existence of the cofibrant-projective
model structure on $\left(\cat V^{\cat M^{\op}}\right)_0$. For a simplicial model category \cat M,
with $\cat V = \sS$ (the category of simplicial sets), this was proven in
\cite[2.8]{Chorny-relative}. For a combinatorial model category $\cat M^{op}$, this
was proven in \cite[3.6]{Duality}.  But the case of contravariant small functors from
a combinatorial model category to \cat V is not covered by the previous results.

\begin{condition}\label{effective}
  Every trivial fibration in \cat V is an effective epimorphism in $\cat V_0$
  (cf., \cite[II, p.~4.1]{Quillen}).
\end{condition}

Basic examples of categories satisfying this condition are (pointed)
simplicial sets, spectra, and chain complexes.

\begin{theorem}\label{tcpms}
  Let \cat V be a closed symmetric monoidal model category satisfying condition
  \ref{effective}, and \cat M a \cat V-model category. The category of small functors
  $\left(\cat V^{\cat M^{\op}}\right)_0$ may be equipped with the cofibrant-projective model structure. Moreover, $\cat V^{\cat M^{\op}}$ becomes a \cat V-model category.
\end{theorem}

\begin{proof}
  Let $I$ and $J$ be the classes of generating cofibrations and generating trivial
  cofibrations in \cat V, and let $\cat M_{\mathrm{cof}}$ denote the subcategory of
  cofibrant objects of \cat M.

  The cofibrant-projective model structure on the category of small presheaves
  $\cat V^{\cat M^{\op}}$  is generated by the following classes of maps.

\[
\cat I = \left\{\left. R_M\otimes i \right|
I\ni i\colon A\cofib B,\, M\in \cat M_{\mathrm{cof}}\right\}\\
\]
and
\[
\cat J = \left\{\left. R_M\otimes j \right|
J\ni j\colon U\trivcofib V,\, M\in \cat M_{\mathrm{cof}}\right\}\\
\]
where $R_M$ is the representable functor $X\mapsto\hom(X,M)\in\cat V$.

It suffices to verify that these two classes of maps admit the generalized
small object argument,\cite{pro-spaces}. More specifically, we need to show that
\cat I and \cat J are locally small. In other words, for any map $f\colon X\to Y$
we need to find a \emph{set} $\cal W$ of maps in \cat I-cof (respectively, \cat J-cof),
such that every morphism of maps $R_M\otimes i \to f$ factors through an
element in $\cal W$. By adjunction, it is sufficient to find a set of cofibrant
objects $\cal U$, such that every map $R_M\to X^A\times_{Y^A} Y^B$ factors
though an element of $\cal U$.

Consider the functor $F=X^A\times_{Y^A} Y^B$. Like any small functor,
$F\colon \cat M^{\op}\to \cat V$ is a left Kan extension from a small full
subcategory $\cat D$ of \cat M, hence a weighted colimit of representable
functors, which, in turn, may be viewed as a coequalizer, by the dual of \cite[3.68]{Kelly}:
\[
F = \int^{\cat D} R_D\otimes FD = \mathrm{coeq}\left( \coprod_{f\colon D'\to D} R_{D'}\otimes FD \rightrightarrows \coprod_D R_D\otimes FD\right).
\]

Therefore, every \cat V-natural transformation $R_A\to F$ in
$\left(\cat V^{\cat M^{\op}}\right)_0$ factors through
$R_{D}\otimes FD$ for some $D\in \cat D$ by the weak Yoneda lemma
for \cat V-categories.  Unfortunately, $R_{D}\otimes FD$ is not
necessarily  \cat I-cofibrant. However, we can find an $\cal
I$-cofibrant object $U$ having a factorization $R_A\to U\to
R_{D}\otimes FD$ for every cofibrant $A\in \cat M$.

Let $q\colon\tilde D \trivfibr D$ be a cofibrant replacement in \cat M, and
$U:=R_{\tilde D}\otimes \tilde D$, with the map $U\to R_{D}\otimes FD$ composed of
$R_D \otimes q$ and $\hom(-,q) \otimes \tilde D$. It suffices to show that the
induced map $\hom(R_A, U)\to \hom(R_A,R_{D}\otimes FD)$ is an epimorphism. This map factors,
in turn, as a composition of two maps
$\hom(A,\tilde D)\otimes \tilde D\to \hom(A, D)\otimes \tilde D\to \hom(A, D)\otimes D$,
each of which is given by tensoring an effective epimorphism with an object of \cat V.
A map is an effective epimorphism if and only if it is the coequalizer of some
pair of parallel maps (see, e.g., the dual of \cite[10.9.4]{Hirschhorn}). Hence,
this is a composition of two effective epimorphisms, and so an epimorphism.
In particular, if $S$ is a unit of \cat V, then
$\hom_{\cat V_0}(S,\hom(R_A, U))\to \hom_{\cat V_0}(S, \hom(R_A,R_{D}\otimes FD))$
is a surjection of sets.

$\cat V^{\cat M^{\op}}$ becomes a \cat V-model category by \cite[Prop.~3.18]{Duality}.
\end{proof}

\begin{definition}\label{F-def}
Consider the following classes of maps in $\cat V^{\cat M^{op}}$:
\[
\cal F_0 = \left\{
\emptyset = \hocolim_\emptyset \emptyset \to R_{\hocolim_\emptyset \emptyset}=R_\emptyset
\right\}
\]
\[
\cal F_1 = \left\{
R_X\otimes A \to R_{X\otimes A} | \, X\in \cat M,\, A\in \cat V\text{ -- cofibrant objects}
\right\}~.
\]
where the map $R_X\otimes A \to R_{X\otimes A}$ is the unit of the adjunction
$Y\colon \cat M \leftrightarrows \cat V^{\cat M^{\op}} :\! -\otimes \Id_{\cat M}$.
%
%
\[
\cal F_2 = \left\{\left.
\hocolim
\vcenter{
    \xymatrix{
    R_A
    \ar[r]
    \ar[d]  & R_B\\
    R_C
    }
}
 \to R_D \right| \,
 \vcenter{
    \xymatrix{
     A \ar[r]\ar[d]  & B \ar[d]\\
     C \ar[r]          & D
    }
}
\vcenter{
    \xymatrix@=0pt{
    \text{ -- homotopy pushout of}\\
    \text{\phantom{x}\hspace{10pt} cofibrant objects in \cat M}
    }
}
\right\}
\]
\[
\cal F_3 = \left\{\left.
\vcenter{
    \xymatrix{
        \hocolim_{k< \kappa}(R_{A_0}\to \ldots\to R_{A_k}\to R_{A_{k+1}}\to\ldots)\ar[d]\\
        R_{\colim_{k < \kappa} A_k }
    }
}
\right|
\vcenter{
    \xymatrix@=0pt{
        A_0\cofib\ldots\cofib A_k \cofib A_{k+1} \cofib \ldots,\\
        A_k \text{ cofibrant for all } k<\kappa
    }
}
\right\}
\]
Set $\cal F:= \cal F_0 \cup \cal F_1 \cup \cal F_2  \cup \cal F_3$.
\end{definition}

\begin{definition}
  Given a class $\cal F$ of natural transformations of cofibrant-projectively cofibrant
  small functors $\cat M^{\op}\to\cat V$ (for $\cat M$ and $\cat V$ as in Theorem \ref{tcpms}),
  we say that a functor $F:\cat C\to\cat V$ is $\cal F$-\emph{local} if every
  $f:G\to H$ in $\cal F$ induces a weak equivalence
  $f^{\ast}\colon \hom(H,F)\to\hom(G,F)$.
\end{definition}

\begin{proposition}\label{F-local}
Let \cat M be a \cat V-model category, and assume that the underlying category of the category of small functors
$\cat V^{\cat M^{\op}}$ may be equipped with the cofibrant projective model structure.
Let $F\in \cat V^{\cat M^{\op}}$ be a cofibrant-projectively fibrant small functor
taking weighted homotopy colimits of cofibrant objects to homotopy limits in \cat V.
Then the functor $F$ is $\cal F$-local (with respect to the class $\cal F$ of maps
from Definition~\ref{F-def}).
\end{proposition}

\begin{proof}
This follows from Yoneda's lemma and the fact that the colimit of a sequence of cofibrations of cofibrant objects is a homotopy colimit.
\end{proof}

\begin{remark}
  We have only included in the class $\cal F$ those morphisms which are required for
  the proof of the inverse implication: $\cal F$-local functors are equivalent to
  the representable functors (see Theorem~\ref{Freyd-rep}). In some situations
  the class $\cal F$ of maps may be reduced even further. For example, if \cat V = \sS,
  then the subclass $\cal F_1$ of maps is redundant, since every weighted homotopy colimit
  in a simplicial category can be expressed in terms of the classical homotopy colimits
  (cf. \cite[Lemma~3.1]{Chorny-BrownRep}).

If $\cat V =\cat M =\Sp$ for some closed symmetric monoidal combinatorial model of spectra,
again $\cal F_1$ is not needed, by to Spanier-Whitehead duality
(see \cite[Lemma~7.2]{Duality}).

However, in general weighted homotopy colimits cannot be expressed in terms of
classical homotopy colimits (that is, homotopy colimits with contractible weight),
as is shown by Luk{\'{a}}{\v{s}}-Vok{\v{r}}{\'{\i}}nek in \cite{weighted}.
\end{remark}

\section{Freyd representability theorem, up to homotopy}\label{Freyd}

\begin{theorem}\label{Freyd-rep}
Let \cat V be a closed symmetric monoidal combinatorial model category satisfying
condition \ref{effective}, and suppose that the domains of the
generating cofibrations of \cat V are cofibrant. If \cat M is a \cat V-model category,
then the underlying category of the category $\cat V^{\cat M^{\op}}$ of small \cat V-presheaves on \cat M
may be equipped with the cofibrant-projective model structure. A small functor $F$
is cofibrant-projectively weakly equivalent to a representable functor if and only
if it takes homotopy colimits of cofibrant objects to homotopy limits.
\end{theorem}

\begin{proof}
  Let $\lambda$ be a regular cardinal such that \cat V is a $\lambda$-combinatorial
  model category. Let $\cal I$ and $\cal J$ be its sets of generating cofibrations and
  trivial cofibrations, respectively. We assume that the domains of the maps in
  $\cal I$ are cofibrant.

  Given a small functor $F\in \cat V^{\cat M^{\op}}$ taking homotopy colimits to
  homotopy limits, let $\widetilde F \trivfibr F$ be a cofibrant replacement for $F$
  in the cofibrant-projective model structure. Then there is a $\lambda$-sequence
  $F_0\to F_1\to \cdots \to F_k\to F_{k+1}\to \cdots \tilde F$ such that
  $F_0(M)=\emptyset$ for all $M\in \cat M$, $\widetilde F=\colim F_k$, and $F_{k+1}$
  is obtained from $F_k$ as a pushout
\[
\xymatrix{
R_{M}\otimes  A
\ar[r]
\ar@{^(->}[d]
                                 & F_k
                                 \ar[d]\\
R_{M}\otimes B
\ar[r]
                                 & F_{k+1},
}
\]
where $M\in \cat M$ is cofibrant, and $(A\cofib B) \in \cal I$, $A$, $B\in \cat V$
are also cofibrant by assumption.

Our proof will proceed by induction. Recall the class $\cal F$ of maps from
Definition~\ref{F-def}, and note that the fibrant replacement of the given functor $F$
is $\cal F$-local by Proposition~\ref{F-local}.

Note also that $F_0=\emptyset$ is $\cal F$-equivalent to
$R_\emptyset=R_{\hocolim_\emptyset \emptyset}$, since the map $\emptyset \to R_\emptyset$
is in $\cal F_0\subset \cal F$.

Suppose by induction that $F_k$ is $\cal F$-equivalent to a representable functor
$R_{X_k}$, where $X_k$ is a fibrant and cofibrant object of \cat M. There is then
a commutative diagram
\[
\xymatrix{
                    &   &   R_{X'_k}
                            \ar@{->>}[dr]^{\dir{~}}\\
R_{M\otimes A}
\ar[ddd]
\ar[rrr]
\ar[urr]
                                 &   &   &   R_{X_k}\\
    & R_{M}\otimes  A
    \ar[r]
    \ar@{^(->}[d]
    \ar[ul]^{\varepsilon}
                                 & F_k
                                 \ar[d]
                                 \ar[ur]
                                 \ar'[u][uu]\\
    & R_{M}\otimes B
    \ar[r]
    \ar[dl]
                                 & F_{k+1}\\
R_{M\otimes B}
}
\]
where the upper horizontal arrow is induced by the universal property of the unit
of the adjunction, and $X'_k$ is a fibrant and cofibrant object of \cat M obtained
as a middle term of the factorization $M\otimes A \cofib X'_k \trivfibr X_k$.

Since the functor $F_k$ is cofibrant, there exists a lift $F_k\to R_{X'_k}$ which is
an $\cal F$-equivalence by the 2-out-of-3 property and does not violate the commutativity
of the above diagram, since the upper slanted arrow $R_{M\otimes A}\to R_{X_k'}$ is also
a natural map induced by the universal property of the unit of adjunction $\varepsilon$.

Let $P:= M\otimes B \coprod_{M\otimes A} X'_k$. We then obtain a commutative diagram
\[
\xymatrix{
R_{M\otimes A}
\ar[ddd]
\ar[rrr]
                                 &   &   &   R_{X_k}'
                                        \ar[ddd]\\
    & R_{M}\otimes  A
    \ar[r]
    \ar@{^(->}[d]
    \ar[ul]
                                 & F_k
                                 \ar[d]
                                 \ar[ur]\\
    & R_{M}\otimes B
    \ar[r]
    \ar[dl]
                                 & F_{k+1}
                                    \ar@{-->}[dr]\\
R_{M\otimes B}
\ar[rrr]
                                &  &  &  R_P
}
\]
in which all the solid slanted arrows are $\cal F$-equivalences and the inner square
is a homotopy pushout. Hence, the homotopy pushout of the outer square
is $\cal F$-equivalent to $F_{k+1}$, and thus the dashed arrow is also an
$\cal F$-equivalence. Finally, let $X_{k+1}$ denote the middle element in the
factorization $X_k' \cofib X_{k+1} \trivfibr \hat P$. Then $R_P\to R_{\hat P}$ is also
an $\cal F$-equivalence, and thus by the 2-out-of-3 property, so  is the lift
$F_{k+1}\to R_{X_{k+1}}$ (which exists since $F_{k+1}$ cofibrant).

If $\kappa$ is a limit ordinal, then $F_{\kappa}=\colim_{k<\kappa} F_k$.
Since this is the colimit of a sequence of cofibrations of cofibrant functors,
$\colim_{k<\kappa} F_k$ is the homotopy colimit $\hocolim_{k<\kappa}F_k$. However, $F_k$ is
$\cal F$-equivalent to $R_{X_k}$. Moreover, by construction there is a sequence of
cofibrations $X_0\cofib\ldots\cofib X_k \cofib X_{k+1} \cofib\ldots$. Hence
$\hocolim R_{X_k}$ is $\cal F$-equivalent to $R_{\colim X_k}$.

If $\kappa$ is not large enough to ensure that  $\colim_{k<\kappa} X_k$ is fibrant,
we can consider the fibrant replacement
$\colim_{k<\kappa} X_k \trivcofib \widehat{\colim_{k<\kappa} X_k}=X_\kappa$.  Combining these
facts together we conclude that $F_\kappa$ is $\cal F$-equivalent to a representable functor
$R_{X_\kappa}$, represented by a fibrant and cofibrant object.

For $\kappa$ large enough we have $F=F_\kappa$, so $F$ is $\cal F$-equivalent to a
functor represented by a fibrant and cofibrant object. But $F$ is an $\cal F$-local
functor by Proposition~\ref{F-local}, and so is $R_{X_\kappa}$ for every $\kappa$.
Hence, $F\simeq R_{X_\kappa}$ for some $\kappa$, since an $\cal F$-equivalence of
$\cal F$-local functors is a weak equivalence.
\end{proof}

The formal category theoretic dual of Theorem \ref{Freyd-rep} is the following:

\begin{theorem}
Let \cat V be a closed symmetric monoidal combinatorial model category satisfying
condition \ref{effective}, and let \cat M be a \cat V-model category such that
the category $\left(\cat V^{\cat M}\right)_0$ may be equipped with the fibrant-projective model structure.
A small functor $F\colon\cat M^{\op}\to\cat V$ is then fibrant-projectively weakly
equivalent to a representable functor if and only if it takes homotopy limits
of fibrant objects to homotopy limits.
\end{theorem}

\section{Brown representability theorem, up to homotopy}\label{Brown}

In this section we prove a homotopy version of the Brown Representability Theorem
for contravariant functors from a locally presentable \cat V-model
category \cat M to \cat V taking homotopy colimits to homotopy limits.

Note that our proof does not use explicitly the presence of compact objects, as most
known proofs in this field do. Rather we show directly that a contravariant homotopy
functor  from \cat M to \cat V is cofibrant projectively weakly equivalent to a small
functor, and then use the Freyd representability theorem, up to homotopy.

\begin{lemma}\label{small}
  Let \cat M be a combinatorial \cat V-category. Then any functor
  $H\colon \cat M^{\op}\to \cat V$ taking homotopy colimits of cofibrant objects to homotopy
  limits is cofibrant-projectively weakly equivalent  to a small homotopy functor
  $F\in \cat V^{\cat M^{\op}}$.
\end{lemma}

\begin{proof}
  Suppose \cat M is a $\lambda$-combinatorial model category, for some cardinal $\lambda$
  such that the weak equivalences are a $\lambda$-accessible subcategory of the
  category of maps of \cat M.

  Consider a functor $H\colon \cat M^{\op} \to \cat V$ taking homotopy colimits of
  cofibrant objects to homotopy limits, with the natural map
\[
H\to F=\Ran_{i\colon \cat M_\lambda \cofib \cat M} i^*H~,
\]
%
where $\Ran$ is the right Kan extension along the inclusion of categories
$i\colon \cat M_\lambda \cofib \cat M$. Here $\cat M_\lambda$ denotes the full subcategory of
$\lambda$-presentable objects in $\cat M$.

This map is a weak equivalence in the cofibrant-projective model category, because
both functors take homotopy colimits to homotopy limits and every cofibrant object
of \cat M is a ($\lambda$-filtered) homotopy colimit of $\lambda$-presentable cofibrant
objects, \cite[Cor.~5.1]{fat}, on which the two functors coincide.

But we can interpret the right Kan extension as a weighted inverse limit of
representable functors
\[
F=\Ran_{i\colon \cat M_\lambda \cofib \cat M} i^*H = \int_{M\in \cat M_\lambda} \hom(H(M),M) = \{ i^*H, i^* Y\},
\]
where $Y\colon \cat M \to \cat V^{\cat M^{\op}}$ is the Yoneda embedding and
$\{ i^*H, i^* Y\}$ is the weighted inverse limit of $i^* Y$ indexed by $i^* H$
(see \cite[3.1]{Kelly}).

Then by the theorem of Day and Lack in \cite{Day-Lack}, $F$ is small as an inverse limit
of small functors.

\end{proof}

\begin{remark}
  When we speak about sufficiently large filtered colimits in a combinatorial
  model category, they turn out to be homotopy colimits, no matter if the objects
  participating in them are cofibrant or not. Therefore the above Lemma admits the following
  concise formulation: any functor taking homotopy colimits to homotopy limits is
  levelwise weakly equivalent to a small functor.
\end{remark}

\begin{theorem}\label{Brown-rep}
  Let \cat V be a closed symmetric monoidal model category satisfying condition
  \ref{effective}, and \cat M a combinatorial \cat V-model category. Any functor
  $H\colon \cat M^{\op}\to \cat V$ taking homotopy colimits of cofibrant objects to
  homotopy limits is then cofibrant-projectively weakly equivalent to a representable functor.
\end{theorem}

\begin{proof}
Follows from Theorem~\ref{Freyd-rep} by Lemma~\ref{small}.
\end{proof}

\section{Counter-example}\label{counter-example}

We have proved so far two representability theorems, up to homotopy. The first is of
Freyd type, i.e., some set theoretical conditions are required to be satisfied by the
functor in question. The second one is of Brown type, i.e., some set theoretical
assumptions apply to the domain category. But what happens if we make no set theoretical
assumptions on either the domain category or the functor? Are exactness conditions
enough to ensure representability up to homotopy?

Mac Lane's classical (folklore) example of a functor $B\colon \Grp \to \Set$,
which assigns to each group $G$ the set of all homomorphisms from the free product
of a large collection of non-isomorphic simple groups to $G$, is an example of a
(strictly) non representable functor. Notice that neither does $B$ satisfy the solution
set condition, nor is the category $\Grp^{\op}$ locally presentable. Perhaps
representability up to homotopy is less demanding and would persist without any conditions?

Our example is similar in nature to Mac Lane's example, but it has also another
predecessor: in \cite{DF}, Dror-Farjoun gave an example of a failure of Brown representability
for generalized Bredon cohomology.

Consider the closed symmetric combinatorial model category of
spaces $\cal S$, with $\cat M = \cal S^{\cal S}$ the category of
small functors. Let $B\colon \cal S\to \cal S$ be a functor which
is not small, such as $B=\hom(\hom(-,S^0), S^0)$. Note that
$B\notin \cat M$, since $B$ is not accessible and all small
functors are. However,  $H=\hom(-,B)\colon \cat M\to \cal S$ is
well defined, since for any small functor $F\in \cat M$, there
exists a small subcategory $i\colon \cat A\hookrightarrow \cal S$
such that $F=\Lan_{i}i^*F$ and hence $H(F) = \hom(F,B) =
\hom_{\cal S^{\cat A}}(i^*F,i^*B)\in \cal S$.

We have defined a functor $H\colon \cat M^{\op}\to \sS$ taking homotopy colimits
of cofibrant objects to homotopy limits, but it is not representable, even up to homotopy:
otherwise, there would exist a small fibrant functor $A\in \cat M$ such that
$H(-)\simeq \hom(-, A)$. By J.~H.~C.~Whitehead's argument we know $A\simeq B$, but
then $B$ would preserve the $\lambda$-filtered homotopy colimits for some $\lambda$.
This is a contradiction, since $B$ does not preserve filtered colimits even of
discrete spaces.

\section{Mapping space recognition principle}\label{recognition}

In this section we assume that the closed symmetric monoidal combinatorial model
category $\cat V$ is the category $\Top_*$ of $\Delta$-generated topological spaces.
This is a locally presentable version of the category of topological spaces, first
proposed by Jeff Smith and described in detail by Fajstrup and Rosicky
in \cite{Fajstrup-Rosicky}.

Let $X\in \Top_*$ be a path-connected space, and let $A\in \Top_*$ be a CW-complex.
We will describe a sufficient condition for there to exist a space $Y\in \Top_*$ such that
$X\simeq \hom(A,Y)$. For example, if $A=S^1$, the required condition is that $X$
can be equipped with an algebra structure over the little intervals operad.
In practice that means that the space $X$ admits $k$-ary operations, i.e.,
maps $X\times \ldots \times X \to X$ satisfying a long list of higher
associativity conditions.

For a space $A$ more general then $S^1$ it is insufficient to consider the structure
given by the maps $X\times \ldots \times X \to X$, by \cite{Badz-Dobr-note}.
We will consider, instead, the structure given by the mutual interrelations of
all possible homotopy inverse limits of $X$ \ -- \ the structure we would have
if $X$ indeed was equivalent to $\hom(A, Y)$ for some $Y$. Indeed, consider
the subcategory of $A$-cellular spaces $C(A)\subset \Top_*$. This is the minimal
subcategory containing $A$ and closed under the homotopy colimits. Any homotopy
colimit of a diagram involving $A$ is then taken by the functor $\hom(-, Y)$ into the
homotopy limit of an opposite diagram involving $X$. In other words, every mapping
space $X$ is equipped with a functor $F_X\colon C(A)\to \Top_\ast$. Moreover,
this functor has a very nice property: it takes homotopy colimits into homotopy limits.

Our goal is to show the converse statement: if for a given $X\in \Top_\ast$ there
exists a functor $F\colon C(A)^\op \to \Top_\ast$ taking homotopy colimits to homotopy
limits and satisfying $F(A)\simeq X$, this $F$ is weakly equivalent to a
representable functor, i.e., there is a space $Y$ such that $F(-)\simeq \hom(-, Y)$.

\begin{theorem}
  For any cofibrant $A\in \Top_*$ and any $X\in \Top_{\ast}$, there is an object
  $Y\in \Top_*$ satisfying $X\simeq \hom(A, Y)$ if and only if there exists a
  functor $F\colon C(A)^{\op}\to \Top_{\ast}$ taking homotopy colimits to homotopy
  limits and satisfying $F(A)\simeq X$.
\end{theorem}

\begin{proof}
Necessity of the condition is clear. We will prove the sufficiency now.

Consider the right Bousfield localization of $\Top_*$ with respect to $A$.
By \cite[5.1.1(3)]{Hirschhorn} we obtain a cofibrantly generated model structure,
which is also combinatorial, since we have chosen to work with a locally presentable
model of topological spaces. We denote the new model category by $\Top_*^A$.
The subcategory of cofibrant objects of $\Top_*^A$ is then $C(A)$ as above.
We denote the inclusion functor by $i\colon C(A)\to \Top_*^A$.

Given a functor $F\colon C(A)\to \Top_*$ taking homotopy colimits to homotopy
limits and satisfying $F(A)=X$, consider the left Kan extension  $H = \Lan_i F$ of $F$
along the inclusion $i$. This functor $H\colon \Top_*^A\to \Top_*$ then satisfies the
condition of Theorem~\ref{Brown-rep}, hence there exists $Y\in \Top_*^A$ such that
$H(-) \simeq \hom(-, Y)$, in particular, $H(A)=F(A)\simeq X \simeq \hom(A,Y).$
\end{proof}

\bibliographystyle{abbrv}
\bibliography{Xbib}

\end{document}